\documentclass[12pt,reqno]{amsart}

\usepackage{times}
\usepackage[T1]{fontenc}
\usepackage{mathrsfs}
\usepackage{latexsym}
\usepackage[dvips]{graphics}
\usepackage[dvips]{graphicx}
\usepackage{epsfig}
\usepackage{hyperref, amsmath,amsfonts,amsthm,amssymb,amscd}
\usepackage{color}

\newtheorem{thm}{Theorem}[section]
\newtheorem{conj}[thm]{Conjecture}
\newtheorem{cor}[thm]{Corollary}
\newtheorem{lem}[thm]{Lemma}
\newtheorem{prop}[thm]{Proposition}

\newtheorem{que}[thm]{Question}

\newcommand{\F}{\mathbb{F}}

\numberwithin{equation}{section}
\newcommand{\con}{\equiv}

\newcommand{\R}{\mathbb{R}}
\newcommand{\Z}{\mathbb{Z}}
\newcommand{\Q}{\mathbb{Q}}
\newcommand{\C}{\mathbb{C}}

\newcommand{\Pj}{\mathbb{P}}
\newcommand{\A}{\mathbb{A}}

\newcommand{\fot}{\frac{1}{3}}

\newcommand{\ol}{\overline}

\newcommand{\comment}[1]{}
\newcommand{\bi}{\begin{itemize}}
\newcommand{\ei}{\end{itemize}}
\newcommand{\ben}{\begin{enumerate}}
\newcommand{\een}{\end{enumerate}}
\newcommand{\be}{\begin{equation}}
\newcommand{\ee}{\end{equation}}
\newcommand{\bea}{\begin{eqnarray}}
\newcommand{\eea}{\end{eqnarray}}
\newcommand{\bal}{\begin{align}}
\newcommand{\eal}{\end{align}}
\newcommand{\ba}{\begin{array}}
\newcommand{\ea}{\end{array}}
\newcommand{\nn}{\nonumber}
\newcommand{\Mod}[1]{\,\,\left(\operatorname{mod}\, #1\right)}

\DeclareMathOperator{\id}{id}

\DeclareMathOperator{\Jac}{Jac}

\newcommand{\simto}{{\hspace{4pt}\to\hspace{-11pt}}^\sim\hspace{7pt}}

\DeclareFontFamily{U}{wncy}{}
\DeclareFontShape{U}{wncy}{m}{n}{<->wncyr10}{}
\DeclareSymbolFont{mcy}{U}{wncy}{m}{n}
\DeclareMathSymbol{\Sha}{\mathord}{mcy}{"58}

\begin{document}

\title[The Arithmetic Geometry of Resonant Rossby Wave Triads]{The Arithmetic Geometry of Resonant Rossby Wave Triads}
\author{Gene S. Kopp}



\date{\today}

\thanks{This research was supported by NSF grant DMS-1401224 and NSF RTG grant 1045119.}

\begin{abstract}  
Linear wave solutions to the Charney-Hasegawa-Mima partial differential equation with periodic boundary conditions have two physical interpretations: Rossby (atmospheric) waves, and drift (plasma) waves in a tokamak.  These waves display resonance in triads.  In the case of infinite Rossby deformation radius, the set of resonant triads may be described as the set of integer solutions to a particular homogeneous Diophantine equation, or as the set of rational points on a projective surface.  We give a rational parametrization of the smooth points on this surface, answering the question: What are all resonant triads?  We also give a fiberwise description, yielding a procedure to answer the question: For fixed $r \in \Q$, what are all wavevectors $(x,y)$ that resonate with a wavevector $(a,b)$ with $a/b = r$?
\end{abstract}

\maketitle

%%%%%%%%%%%%%%%%%%%%

\section{Introduction}

This paper determines the set of resonant Rossby/drift wavevector triads as values of rational functions of three parameters.  It uses methods from algebraic geometry and number theory.  We identify the primitive resonant triads as rational points on a rational elliptic surface, and we prove results about the group structure on the fibers (which are elliptic curves).

\subsection{Rossby Waves}

Atmospheric Rossby waves are large-scale meanders in high-altitude winds caused by a planet's rotation.  They are a major influence on the weather and have been widely studied.  Recently, Petoukhov, Rahmstorf, Petri, and Schellnhuber \cite{PRPS} implicated quasiresonance of Rossby waves in several extreme weather events.  Coumou, Lehmann, and Beckmann \cite{CLB} connected the amplification of Rossby waves with global warming.

Mathematically, Rossby waves are modelled as linear wave solutions to a nonlinear differential equation---the Charney-Hasegawa-Mima equation (CHME) \cite{BH}---expressing conservation of potential vorticity (see \cite{Ped}, especially section 3.16).  
Rossby waves exist in the $\beta$-plane model, 
in which the Coriolis force varies as a linear function of the latitude (see \cite{Ped}, section 3.17 and chapter 6).  The $\beta$-plane model was introduced by Rossby \cite{Rossby}, and the CHME was formally derived in the meteorological context by Charney \cite{Char1, Char2}.  
Many years later, Hasegawa and Mima rediscovered the CHME in a plasma physics context \cite{HM1,HM2}: Rossby waves in the $\beta$-plane model are mathematically equivalent to drift waves in a plasma in a tokamak.

The shape of a Rossby/drift wave is described by its zonal and meridianal wavenumbers, represented as an integer \textit{wavevector} $(k,\ell) \in \Z^2$.  The zonal (east/west) wavenumber $k$ is the number of crests per period (around a latitudinal circle), whereas the meridianal (north/south) wavenumber $\ell$ quantifies its meridianal propagation.

Triples of Rossby waves may display resonance, in which case they are called a \textit{resonant triad}.  Under certain physical conditions (large $\beta$), exact resonances dominate the behavior of a system obeying the CHME, by a theorem of Yamada and Yoneda \cite{YY}.

Enumerating resonant triads up to some large wavenumber bound is a question of practical importance. As Bustamante and Hayat \cite{BH} stress in their introduction, information about high-frequency resonances is needed to describe small-scale turbulence in numerical simulations.  Moreover, mesoscale waves observed in Jupiter's atmosphere have wavelenghts of about 20 km, and would have wavenumbers of approximately 20000 if ``extended to cover a significant part of Jupiter's surface'' \cite{BH}\footnote{See the NASA photo at \url{http://photojournal.jpl.nasa.gov/catalog/PIA00724}.  Jupiter's mean radius $r$ is about 70000 km, and the waves were observed propagating east-west at $15^\circ$ South.  So, the zonal wavenumber would be $2\pi r \cos\left(15^\circ\right)/(20{\rm \ km}) \approx 20000$ waves per period (or about 3000 waves per radian).}.

The set of resonant triads may be described as the set of integer solutions to a Diophantine equation, as explained in the next section.  In this paper, we give two descriptions of the solutions to this equation: a rational parametrization and an elliptic fibration.

\subsection{The CHME and Rossby Waves}

The Charney-Hasegawa-Mima equation is\footnote{Sometimes an additional term $U\frac{\partial}{\partial x}\left(\Delta \psi\right)$ for a mean zonal flow is included; including it would have no effect on the Diophantine equation we'll obtain for resonant wavevectors.}
\begin{equation}\label{CHMEgen}
\frac{\partial}{\partial t}\left(\Delta \psi - F \psi \right) + \beta \frac{\partial \psi}{\partial x} + [\psi, \Delta \psi] = 0,
\end{equation}
where $\psi = \psi(x,y,t)$ is a function on $\R^2 \times \R$.
This first two terms are linear, and their sum is called the linear part.  The last term is nonlinear.  Standard notation is used for the Euclidean Laplacian and the Poisson bracket:
\begin{align}
\Delta &:= \left(\frac{\partial}{\partial x}\right)^2 + \left(\frac{\partial}{\partial y}\right)^2,\\
[f,g] &:= \frac{\partial f}{\partial x}\frac{\partial g}{\partial y} - \frac{\partial f}{\partial y}\frac{\partial g}{\partial x},
\end{align}
and $\beta > 0$ and $F \geq 0$ are constants.\footnote{For atmospheric Rossby waves, $\beta = \frac{2 \varpi \cos \theta}{r}$, where $\varpi$ is the planet's angular rotation, $\theta$ is the latitude, and $r$ is the planet's mean radius.  The constant $F=\frac{1}{R^2}$, where $R$ is the \textit{Rossby deformation radius}.}  The physical meanings of the parameters are: $(x,y)$ are spacial coordinates, $t$ is a temporal coordinate, and $\psi$ the streamfunction of an incompressible flow.

For the derivation of the CHME in both the meteorology and plasma physics contexts, see \cite{Has}.

We will only consider the limiting case of (\ref{CHMEgen}) to infinite Rossby deformation radius.  Equivalently, we set $F=0$.  So, the equation we will consider is 
\begin{equation}\label{CHME}
\frac{\partial}{\partial t}\left(\Delta \psi\right) + \beta \frac{\partial \psi}{\partial x} + [\psi, \Delta \psi] = 0.
\end{equation}

The equation (\ref{CHME}) has a rich set of solutions, and no general solution is known.  However, one simple family of solutions, called \textit{Rossby waves}, are given by exponential functions.\footnote{The real part of $\psi_{k,\ell}$ is also a solution, and, in some applications, may instead be called a Rossby wave.}
\begin{align}
\psi &= \psi_{k,\ell} = \exp\left(i(kx+\ell y-\omega t)\right)\\
\omega &= \omega_{k,\ell} = -\frac{\beta k}{k^2+\ell^2}.
\end{align}
Indeed, Rossby waves satisfy both the linear and nonlinear parts of (\ref{CHME}), separately.
The vector $(k,\ell) \in \R^2$ is called the \textit{wavevector}, $k$ is the \textit{zonal wavenumber}, and $\ell$ is the \textit{meridional wavenumber}.

\subsection{Resonant Triads}
Generally, linear combinations of Rossby waves no longer satisfy the nonlinear part of the CHME.  However, under certain conditions, triads $\{(k_1,\ell_1),(k_2,\ell_2),(k_3,\ell_3)\}$ of Rossby waves ``resonate'' so that the real part of a certain linear combination
\begin{equation}\label{triad}
\Re\left(A_1(\tau)\psi_{k_1,\ell_1} + A_2(\tau)\psi_{k_2,\ell_2} + A_3(\tau)\psi_{k_3,\ell_3}\right),
\end{equation}
where $\tau$ is a slow time scale,
satisfies the CHME.  See \cite{Lynch} and especially the supplement \cite{LynchSupp} for details, and \cite{Ped, Nay, Cra} for additional background.

Resonant triads of Rossby waves satisfy the conditions
\begin{align}
k_1 + k_2 &= k_3,\\
\ell_1 + \ell_2 &= \ell_3,\\
\omega_1 + \omega_2 &= \omega_3.
\end{align}
We have $(k_2,\ell_2)=(k_3-k_1,\ell_3-\ell_1)$, so the condition for $\{(k_1,\ell_1),  (k_2,\ell_2), (k_3,\ell_3)\}$ to be a resonant triad comes down to the condition on the angular frequencies:
\begin{align}\label{omega}
\omega(k_1,\ell_1) + \omega(k_3-k_1,\ell_3-\ell_1) &= \omega(k_3,\ell_3),
\end{align}
which we rewrite as
\begin{equation}
\frac{k_1}{k_1^2+\ell_1^2}-\frac{k_3}{k_3^2+\ell_3^2}-\frac{k_1-k_3}{(k_1-k_3)^2+(\ell_1-\ell_3)^2}=0.
\end{equation}

\subsection{CHME on a Torus}

We impose the additional requirement that the solutions to the CHME are periodic in both spatial coordinates: 
\begin{align}
\psi(x+2\pi,y,t) &= \psi(x,y,t),\\
\psi(x,y+2\pi,t) &= \psi(x,y,t).
\end{align}
Equivalently, rather than considering the CHME on $\R^2 \times \R$, we can specify that the domain is $\left(\R/2\pi\Z\right)^2 \times \R$.  
Periodicity implies that the wavenumbers are integers.

Several authors framed the problem of enumerating resonant triads as a number-theoretic question \cite{BH, KY}.  For convenience, change notation and set $(k_1,\ell_1,k_3,\ell_3)=(a,b,x,y)$.
\begin{que}
Find the integer solutions $(a,b,x,y)$ to the Diophantine equation
\begin{equation}\label{eqn}
\frac{a}{a^2+b^2}-\frac{x}{x^2+y^2}-\frac{a-x}{(a-x)^2+(b-y)^2}=0.
\end{equation}
\end{que}
Bustamante and Hayat \cite{BH} classify integer solutions to (\ref{eqn}) by mapping them bijectively to representations of zero by a certain quadratic form, yielding an algorithm for enumerating all resonant triads (see further discussion in section \ref{seccompiss}).  Kishimoto and Yoneda \cite{KY} show that there are no solutions where $b=0$.
Both sets of authors also draw attention to a special family of resonant triads (called ``pure cube triads'' by \cite{BH}): 
\begin{align}\label{u=0}
(a,b) &= (s^4, s t^3)\\
(x-a,y-b) &= (t^4-s^4, -s^3 t - s t^3)\\
(x,y) &= (t^4, -s^3 t),
\end{align}  
where $s,t \in \Z$.

Other objects of interest to \cite{KY} and \cite{YY} are the \textit{wavevector set} $\Lambda$ and \textit{primitive wavevector set} $\Lambda'$, and their distribution within the integer lattice.
\begin{align}
\Lambda &= \{(a,b) \in \Z^2 : \exists (x,y) \in \Z^2, (a,b,x,y) \in X, ax(a-x)\neq 0\},\label{Lambdadef}\\
\Lambda' &= \{(a,b) \in \Z^2 : \exists (x,y) \in \Z^2, (a,b,x,y) \in X, ax(a-x)\neq 0, \gcd(a,b,x,y)=1\}.\label{Lambdaprimedef}
\end{align}
The condition $ax(a-x) \neq 0$ imposed in (\ref{Lambdadef}--\ref{Lambdaprimedef}) rules out single wave solutions and zonal resonances (see also the discussion at the end of section \ref{fibration1}).

\subsection{Results}

The left hand side of (\ref{eqn}) is a homogeneous rational function in $(a,b,x,y)$, so any rational solution $(a_0,b_0,x_0,y_0)$ gives rise to a primitive integer solution $(\lambda a_0,\lambda b_0,\lambda x_0,\lambda y_0)$ by clearing denominators and dividing out common factors.  A \textit{primitive} integer solution is one where $\gcd(a,b,x,y)=1$, and any integer solution is an integer multiple of a primitive integer solution.  The problem of enumerating primitive integer solutions to (\ref{eqn}) is the same as the problem of enumerating rational solutions up to a common factor, that is, rational solutions in projective space.

We may clear denominators in (\ref{eqn}) and do some algebra to obtain the equation
\begin{equation}\label{eqn2}
x (a^2 + b^2) (a^2 + b^2 - 2 a x - 2 b y) = a (x^2 + y^2) (x^2 + y^2 - 2 a x - 2 b y).
\end{equation}
We consider (\ref{eqn}) and (\ref{eqn2}) as equations in projective space $\Pj^3$ with homogeneous coordinates $[a:b:x:y]$.  
Let $X$ be the locus of solutions to (\ref{eqn2}) in $\Pj^3$; $X$ is a degree $5$ surface.  Let $X^\circ \subset X$ be the locus of solutions to (\ref{eqn}).  

Finding points on $X$ is essentially the same problem as finding points on $X^\circ$; $X$ has a few extra points that are easy to describe (see section \ref{exceptsec}).  

For fixed $(a,b) \in \C^2 \setminus (0,0)$, let $C(a,b)$ be the locus of (\ref{eqn2}) in affine space $\A^2$.  Over any field containing $a$ and $b$, $C(a,b) \cong C(\lambda a, \lambda b)$ by $(x,y) \mapsto (\lambda x, \lambda y)$, so we think of $C(a,b)$ as depending only on $[a:b] \in \Pj^1$.  However, we will also want to consider, for $(a,b) \in \Z^2$, the integer points on $C(a,b)$, and a scaling factor may enlarge the set of integer points.

We show that $X$ is a \textit{rational surface}, that is, a surface with a rational parametrization (meaning, a birational isomorphism with $\Pj^2$ defined over $\Q$).  We give the following parametrization, which may be used to directly enumerate integer points on $X$.
\begin{thm}\label{paramthm}
The surface $X$ is a rational surface.  A rational parametrization $\Pj^2 \to X$, where
$[s:t:u] \mapsto [a:b:x:y]$ are homogeneous coordinates, is given by
\begin{align}\label{param}
\left[\begin{array}{c}
a \vspace{-7pt} \\  .. \vspace{-1pt} \\  b \vspace{-7pt} \\ .. \vspace{-1pt} \\ x \vspace{-7pt} \\ .. \vspace{-1pt} \\ y
\end{array}\right]
&=
\left[\begin{array}{c}
s^3t(s-2u) \vspace{-7pt} \\  .. \vspace{-1pt} \\  s(-s^2u(s-2u)+(t^2+u^2)(t^2-2su+u^2)) \vspace{-7pt} \\ .. \vspace{-1pt} \\ t(t^2+u^2)(t^2-2su+u^2) \vspace{-7pt} \\ .. \vspace{-1pt} \\ (t^2+u^2)(-s^2(s-2u)+u(t^2-2su+u^2)) 
\end{array}\right].
\end{align}
Its rational inverse $X \to \Pj^2$, $[a:b:x:y] \mapsto [s:t:u]$ is given by
\begin{align}
\left[\begin{array}{c}
s \vspace{-7pt} \\  .. \vspace{-1pt} \\  t \vspace{-7pt} \\ .. \vspace{-1pt} \\ u
\end{array}\right]
&=
\left[\begin{array}{c}
a^2+b^2 \vspace{-7pt} \\  .. \vspace{-1pt} \\  bx-ay \vspace{-7pt} \\ .. \vspace{-1pt} \\ ax+by
\end{array}\right].
\end{align}
\end{thm}

The surface $X$ may be described more precisely as a rational \textit{elliptic surface},\footnote{The definition of an elliptic surface is usually taken to include the conditions that the surface is smooth and the generic fiber is smooth.  We do \textit{not} assume either of these conditions; indeed, $X$ is non-smooth and $C(a,b)$ is always non-smooth.} that is, a surface with a map $X \to \Pj^1$ so whose generic fiber has geometric genus 1, along with an \textit{identity section} $\Pj^1 \to X$. The integer points $(x,y)$ on a fiber $C(a,b)$ are the wavenumbers that exhibit resonance with $(a,b)$.

\begin{thm}\label{ellthm}
$X$ is a (singular) rational elliptic surface.  The map $X \to \Pj^1$ given in homogeneous coordinates by $[x:y:a:b] \mapsto [a:b]$ (and having fibers $C(a,b)$), along with the choice of identity section $[a:b] \mapsto [0:0:a:b]$, is an elliptic fibration.  The singular fibers over $\ol{\Q}$ occur at $[a:b] = [0:1], [\pm i : 1], [\pm 2i : 1]$.  A Weierstrass form of (a smooth model for) $C(a,b)$ is
\begin{align}\label{niceform}
W^2=Z^3+(a^2-2b^2)Z^2+(a^2+b^2)^2Z.
\end{align}
\end{thm}

The rational points on an elliptic curve have the structure of a finitely generated abelian group, the Mordell-Weil group. 
We prove a result about the structure of this group for $C(a,b)$.

\begin{thm}\label{rank}
The Mordell-Weil group of $C(a,b)$ is of the form $\Z/2\Z \times \Z^r$, where the rank $r \geq 1$ depends on $(a,b)$.  The point $P=(a,b)$ has order $2$ and generates the torsion subgroup, and the point $Q=(0,2b)$ has infinite order.
\end{thm}

A ``typical'' elliptic curve over $\Q(t)$ has trivial Mordell-Weil group, that is, rank 0 and no torsion.  Thus, Theorem \ref{rank} says that $X$ is a special member of the class of elliptic surfaces.

For definitions and background information about elliptic curves and elliptic surfaces, see Silverman's books \cite{Sil1,Sil2}.

\subsection{Discussion}\label{secdisc}

This paper describes the surface $X$ in two different ways.  The first description is a parametrization of $X$ as a rational surface.  This parametrization gives a procedure for enumerating points on $X$ (i.e., primitive resonant Rossby triads).  The algorithmic performance of this enumeration procedure is discussed in the next section.

The second description is a fiberwise description: It describes all the triads with $a/b$ fixed.  For example, the resonant wavevectors with $a/b=1$---that is, those whose wave packets propagate due northward\footnote{The group velocity of $\psi_{a,b}$ is $C_{g} := \frac{\partial \omega}{\partial a} = \left(\frac{\beta\left(a^2-b^2\right)}{\left(a^2+b^2\right)^2},\frac{2\beta ab}{\left(a^2+b^2\right)^2}\right)$, so when $a/b=1$, the zonal component is zero and the meridional component is positive.}---are the rational points on $C(1,1)$, a rank 1 elliptic curve.  This special case is discussed further in section \ref{egz}.

Our parametrization of $X$ generalizes the family of points \ref{u=0} by introducing a third variable, $u$.  Like Bustamante and Hayat's classification (which involves a different rational change of variables), our parametrization provides a conjecturally fast procedure for enumerating points on $X$.  Computational issues are discussed in \ref{seccompiss}.

Our fiberwise description of $X$ is new and yields new information about the primitive wavevector set, shown in Figure \ref{lambdaprime}.  Specifically, the nonzero rank of $C(a,b)$ implies that there are infinitely many points on any line through the origin with rational slope.
In other words, for any nonzero rational number $r$, there are infinitely many primitive resonant wavevectors $(x,y) \in \Lambda'$ such that $\frac{x}{y} = r$.  Equivalently, there are infinitely many primitive resonant wavevectors with any southward phase direction of rational slope.

\begin{figure}[h]
\begin{center}
\includegraphics[scale=0.5]{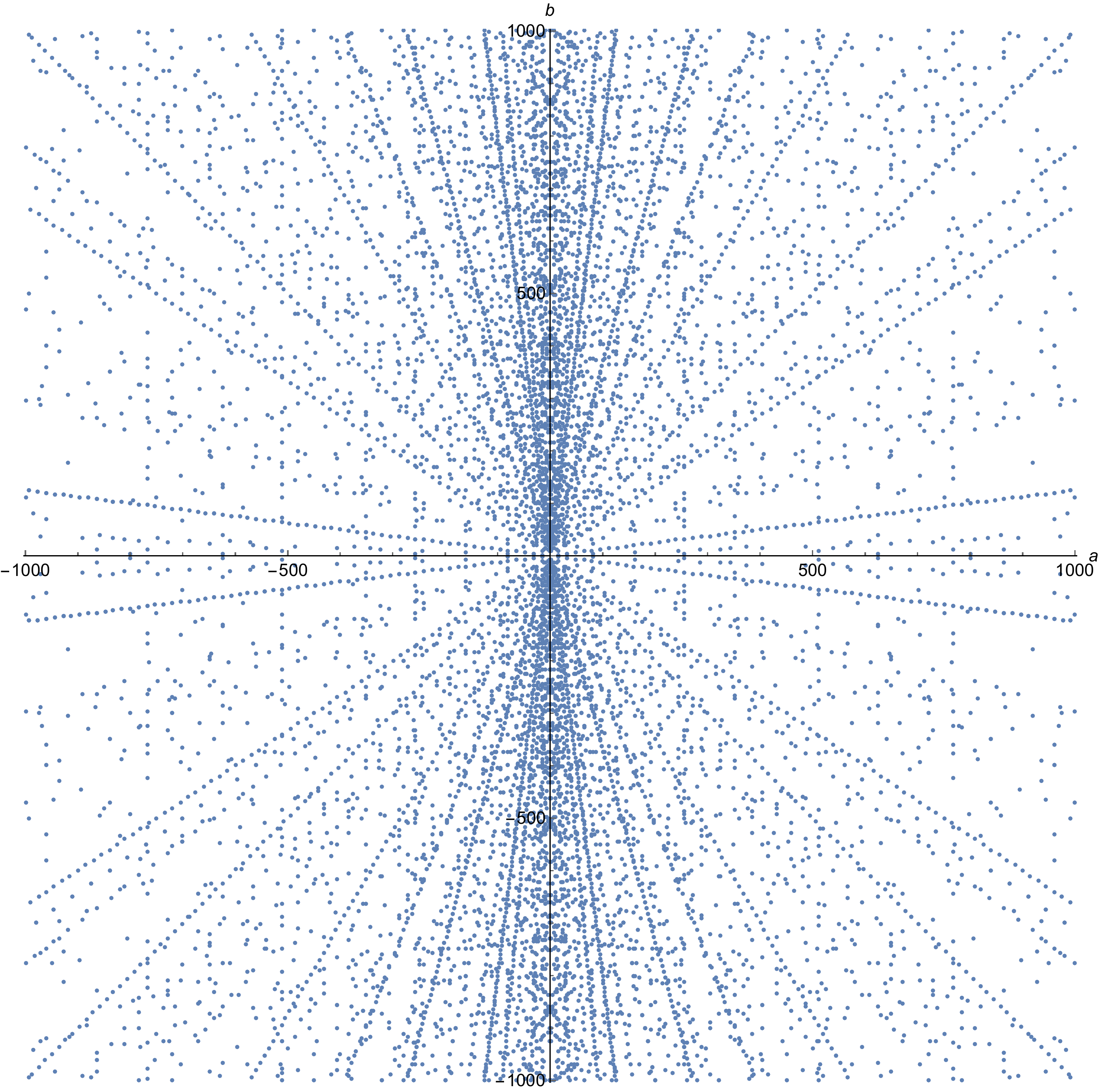}
\end{center}
\caption{The wavevector set $\Lambda$ on the domain $|a|,|b| \leq 1000$.}
\label{lambda}
\end{figure}

\begin{figure}[h]
\begin{center}
\includegraphics[scale=0.5]{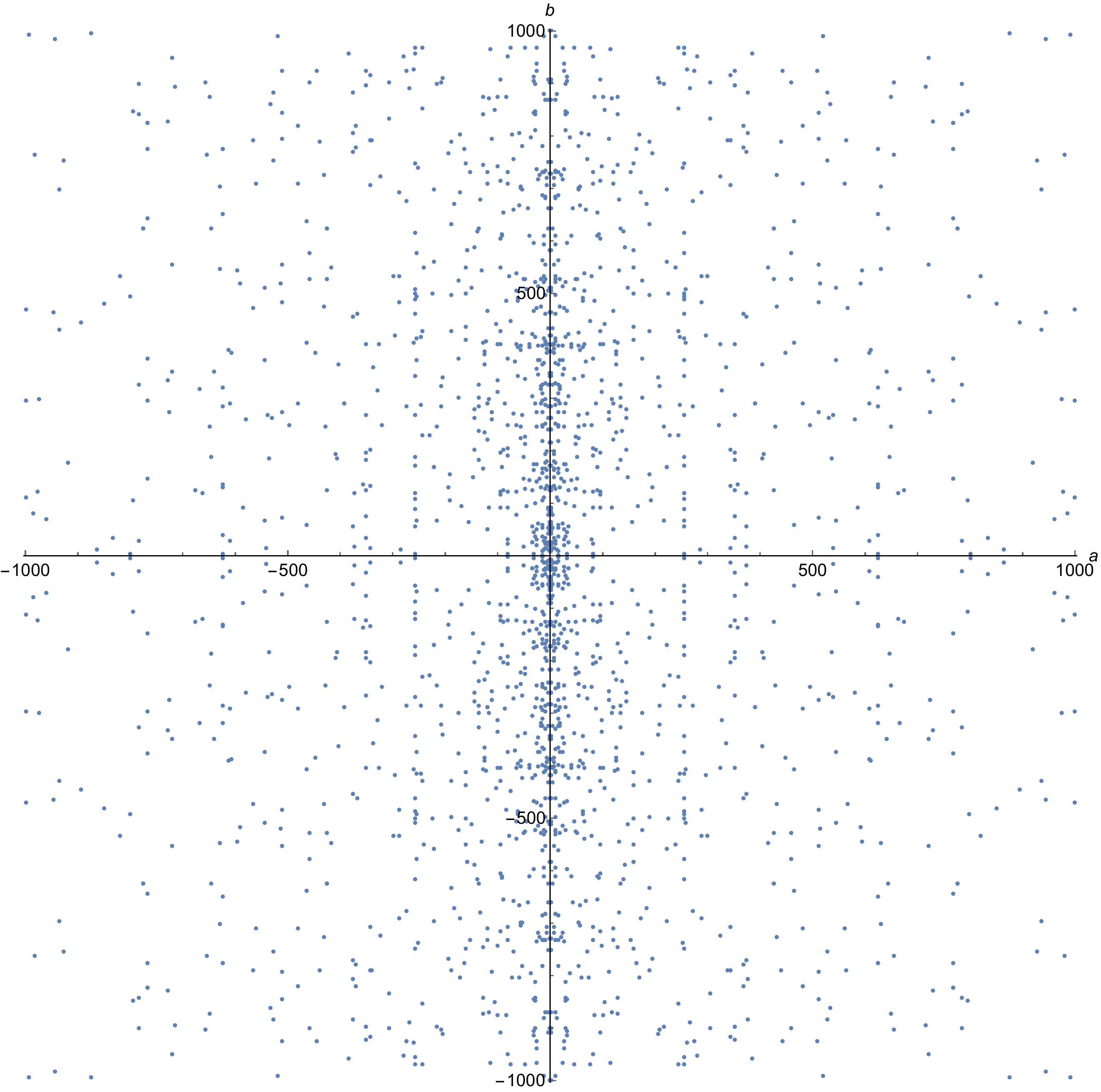}
\end{center}
\caption{The primitive wavevector set $\Lambda'$ on the domain $|a|,|b| \leq 1000$.}
\label{lambdaprime}
\end{figure}

\subsection{Conjectures and Further Work}

Kishimoto and Yoneda \cite{KY} remark that the wavenumber set appears highly anisotropic.  Specifically, this set becomes sparser faster as $a \to \infty$ than as $b \to \infty$.  Based on numerical observations, we suggest some stronger conjectures about the anisotropy.

Numerical evidence suggests that the horizontal lines contain only finitely many points in the wavenumber set.

\begin{conj}\label{conj1}
If $(a,b)$ is in the wavenumber set $\Lambda$ and $a,b > 0$, then $a<\sqrt{3}b^4$ unless $(a,b)=(800,4)$.  In particular, there are finitely many points on every horizontal line in Figures \ref{lambda} and \ref{lambdaprime}.  In other words, there are only finitely many Rossby waves with fixed meridional wavenumber that exhibit nonzonal resonances.
\end{conj}

On the other hand, it's easy to produce infinitely many points on some vertical lines: For example, if $a=s^4$, then $(a,st^3) \in \Lambda$ for any $t \in \Z$.  We conjecture, based on numerical computations, that there are actually infinitely many points on any vertical line (except the $b$-axis) in Figures \ref{lambda} and \ref{lambdaprime}.

\begin{conj}\label{conj2}
For fixed nonzero $a \in \Z$, $\{b \in \Z : (a,b) \in \Lambda'\}$ is infinite.
\end{conj}

\noindent Define the growth functions
\begin{align}
F_1(N) &:= \left|\Lambda' \cap [-N,N]^2\right|\\% &\sim c_1 N\\
F_2(N) &:= \left|\Lambda' \cap (\R\times[-N,N])\right|\\% &\sim c_2 N\\
F_3(N) &:= \left|\{(a,b,x,y)\in X : \gcd(a,b,x,y) = 1, (a,b,a-x,b-x,x,y) \in [-N,N]^6\}\right|% &\sim c_3 N
\end{align}
Numerical evidence suggests that $F_1(N)$, $F_2(N)$, and $F_3(N)$ all grow linearly in $N$, to first order.  Thus, the corresponding quantities without the primitivity condition should grow like $\Theta(N\log N)$.  We hope to investigate the growth of these quantities in a future paper.

Finally, we make a remark about the distribution of ranks of these elliptic curves.  For $(a,b) \in [1,100]^2$, we found (using Magma) that 39\% of the $C(a,b)$ had rank 1, 46\% had rank 2, 13\% had rank 3, and 1\% had rank 4.  This looks similar to numerical data number theorists have obtained for other rank 1 families.  However, the widely believed Minimalist Conjecture predicts that, in the limit as $N \to \infty$, 50\% of the $C(a,b)$ in $[1,N]^2$ should have rank 1, and 50\% rank 2.  See \cite{Wei} for background about this problem and extensive numerical computations.

\subsection{Computational Issues}\label{seccompiss}

Bustamante and Hayat's classification \cite{BH} yields an algorithm for enumerating the set of nontrivial resonant triads.  They identifying this set with the zero-set of a diagonal quadratic form in four variables, and enumerate that zero-set using classical theorems on representations of integers as $n=x^2+dy^2$ for $d=1,3$.  Bustamante and Hayat pose the challenge of generating all or most primitive triads within a box of size $N$:
\begin{equation}
B_N := \{(a,b,x,y)\in X : \gcd(a,b,x,y) = 1, (a,b,a-x,b-x,x,y) \in [-N,N]^6\}
\end{equation}
For $N=5000$, Bustamante and Hayat's method produced $870 \times 12 = 435 \times 24$ triads in 20 minutes in Mathematica on an ``8-core desktop'' \cite{BH}.  Bustamante and Hayat compare this to a na\"{i}ve search, which would take about 15 years to complete, although Kartashova and Kartashova \cite{KK} remark that a smart search (confined to an open neighborhood of $X$) would be faster.\footnote{Kartashova and Kartashova describe, in section IV of \cite{KK}, a search algorithm enumerating $B_N$ in $O(N^3)$ time, compared to $O(N^4)$ for a na\"{i}ve search.}

For our method, we use the parametrization (\ref{param}) to enumerate points on $X$, with a slight modification.  We changed coordinates on $\Pj^2$ to make it easier to mod out by symmetries: We set $w=s-2u$ and use the coordinates $[s:t:w]$.  Then all nontrivial resonant triads are generated by $s>0$, $t>0$, $w>0$, $t^2 > \frac{1}{4}(3s-w)(s+w)$, and by the action of the group of order 24 generated by the symmetries
\begin{align*}
(a,b,x,y) &\mapsto (-a,-b,-x,-y),\\
(a,b,x,y) &\mapsto (x,y,a,b),\\
(a,b,x,y) &\mapsto (a,b,a-x,b-y),\\
(a,b,x,y) &\mapsto (a,-b,x,-y).
\end{align*}

Searching $0 < s,t,w \leq 200$ and applying symmetries took about 5 minutes of CPU time and found $443 \times 24$ triads in $B_{5000}$.  Searching $0 < s,t,w \leq 500$ and applying symmetries took about 80 minutes and found $463 \times 24$ triads in $B_{5000}$.
Computations were performed in Mathematica on a MacBook Pro laptop (2.9 GHz dual-core i7).

We have not established any upper bound on the search space necessary to output all nontrivial resonant triads up to a given wavenumber bound.  We hope to establish such a bound in a future paper.

%%%%%%%%%%%%%%%%%%%%

\section{The Surface of Rossby Triads}\label{surface}

\subsection{Exceptional Sets}\label{exceptsec}

Our object is to determine the rational points of $X^\circ$---that is, the primitive Rossby wave triads.  It's convenient to study the Zariski closure $X$, which will have a few extra points.  Before we do so, it's important to quantify the difference between these two sets.  We also need to understand the singular points of $X$, because the birational map from $X \to \Pj^2$ we will define in section \ref{paramsec} will be undefined at these points (and well-defined everywhere else).

\begin{prop}
$X \setminus X^\circ$ is a union of 3 real and 12 nonreal lines; specifically,
\begin{align}
\hspace*{-.05in}
X \setminus X^\circ = &
        \{a=b=0\} \cup \{x=y=0\} \cup \{a-x=b-y=0\} 
\cup \{a^2+b^2=x^2+y^2=0\} \nn \\ &\cup \{a^2+b^2=(a-x)^2+(b-y)^2=0\} \cup \{x^2+y^2=(a-x)^2+(b-y)^2=0\}.
\end{align}
\end{prop}
\begin{proof}
On $X \setminus X^\circ$, we have (\ref{eqn2}), and one of the denominators in (\ref{eqn}) vanishes.

If $a^2+b^2=0$, then $a(x^2+y^2)(x^2+y^2-2ax-2by)=0$.
\begin{itemize}
\item If $a=0$, then $b^2=0$, so $a=b=0$.
\item If $x^2+y^2=0$, then $a^2+b^2=x^2+y^2=0$; this is a union of the four nonreal lines $\{a+bi=x+yi=0\}\cup\{a+bi=x-yi=0\}\cup\{a-bi=x+yi=0\}\cup\{a-bi=x-yi=0\}$.
\item If $x^2+y^2-2ax-2by=0$, then $a^2+b^2=(a-x)^2+(b-y)^2=0$; this is a also union of four nonreal lines.
\end{itemize}

The two remaining cases, $x^2+y^2=0$ and $(a-x)^2+(b-y)^2=0$, are similar.
\end{proof}

The real points of $X \setminus X^\circ$ correspond to ``resonant triads'' that are actually just single Rossby waves.  If, for example, $x-a=y-b=0$, then the triad is $\{(a,b),(0,0),(a,b)\}$, and the conditions on the $A_i(t)$ (see equation (5) in \cite{LynchSupp}) reduce to $A_1'(t)=A_3'(t)=0$.  Thus, \ref{triad} becomes $\psi = \Re\left(C\psi_{a,b}\right)$, the real part of a single Rossby wave.

\begin{prop}
The singular locus $X^{\rm sing}$ of $X$ is the algebraic set where
\begin{align}
a^2+b^2=0 \mbox{ and } x^2+y^2=0 \mbox{ and } ay-bx=0,
\end{align}
that is, the two complex projective lines
\begin{align}
[a : ia : x : ix] \mbox{ and } [a : -ia : x : -ix].
\end{align}
\end{prop}
\begin{proof}
The surface $X = \{[a:b:x:y]\in\Pj^3 : f(a,b,x,y)=0\}$, where
\begin{equation}\label{f}
f(a,b,x,y) = x (a^2 + b^2) (a^2 + b^2 - 2 a x - 2 b y) - a (x^2 + y^2) (x^2 + y^2 - 2 a x - 2 b y).
\end{equation}
The singular locus is 
\begin{equation}
X^{\rm sing} = \left\{[a:b:x:y]\in\Pj^3 : f=\frac{\partial f}{\partial a}=\frac{\partial f}{\partial b}=\frac{\partial f}{\partial x}=\frac{\partial f}{\partial y}=0\right\}.
%X^{\rm sing} = \left\{[a:b:x:y]\in\Pj^3 : f(a,b,x,y)=\frac{\partial f}{\partial a}(a,b,x,y)=\frac{\partial f}{\partial b}(a,b,x,y)=\frac{\partial f}{\partial x}(a,b,x,y)=\frac{\partial f}{\partial y}(a,b,x,y)=0\right\}.
\end{equation}
A straightforward algebraic computation shows that this is the set described in the proposition.
\end{proof}

Note that $X^{\rm sing}$ is contained in $X \setminus X^\circ$; that is, $X^\circ$ is everywhere smooth.

\subsection{Rational Parametrization}\label{paramsec}

We will now prove Theorem \ref{paramthm}, which we restate here for convenience.

\begin{thm}%\label{paramthm}
The surface $X$ is birational to $\Pj^2$ over $\Q$.  A rational parametrization $\Pj^2 \to X$, in homogeneous coordinates $[s:t:u] \mapsto [a:b:x:y]$, is given by
\begin{align}\label{paramagain}
\left[\begin{array}{c}
a \vspace{-7pt} \\  .. \vspace{-1pt} \\  b \vspace{-7pt} \\ .. \vspace{-1pt} \\ x \vspace{-7pt} \\ .. \vspace{-1pt} \\ y
\end{array}\right]
&=
\left[\begin{array}{c}
s^3t(s-2u) \vspace{-7pt} \\  .. \vspace{-1pt} \\  s(-s^2u(s-2u)+(t^2+u^2)(t^2-2su+u^2)) \vspace{-7pt} \\ .. \vspace{-1pt} \\ t(t^2+u^2)(t^2-2su+u^2) \vspace{-7pt} \\ .. \vspace{-1pt} \\ (t^2+u^2)(-s^2(s-2u)+u(t^2-2su+u^2)) 
\end{array}\right].
\end{align}
Its rational inverse $X \to \Pj^2$, $[a:b:x:y] \mapsto [s:t:u]$ is given by
\begin{align}
\left[\begin{array}{c}
s \vspace{-7pt} \\  .. \vspace{-1pt} \\  t \vspace{-7pt} \\ .. \vspace{-1pt} \\ u
\end{array}\right]
&=
\left[\begin{array}{c}
a^2+b^2 \vspace{-7pt} \\  .. \vspace{-1pt} \\  bx-ay \vspace{-7pt} \\ .. \vspace{-1pt} \\ ax+by
\end{array}\right].
\end{align}
\end{thm}
\begin{proof}
Define the rational maps $\phi : \Pj^3 \to \Pj^3$ and $\psi : \Pj^3 \to \Pj^3$ by
\begin{align}\label{phi}
\phi([s:t:u:v]) &= [st^2 : s(sv-ut) : stv : -t^3-tu^2+suv], \\
\psi([a:b:x:y]) &= [a(a^2 + b^2) : a(b x - a y) : a(a x + b y) : x(b x - a y)].
\end{align}
The map $\phi$ is defined away from the four lines 
\begin{equation}
\{s=t=0\}\cup\{s=0,\, t=iu\}\cup\{s=0,\, t=-iu\}\cup\{t=v=0\}.
\end{equation}
The map $\psi$ is defined away from the four lines
\begin{equation}
\{a=b=0\}\cup\{a=x=0\}\cup\{a=ib,\, x=iy\}\cup\{a=-ib,\, x=-iy\}
\end{equation}

On a Zariski dense open set, $\phi \circ \psi = \id$ and $\psi \circ \phi = \id$; that is, $\phi$ and $\psi$ are rational inverses.
Moreover, if $[a:b:x:y]=\phi([s:t:u:v])$ and $[a:b:x:y] \in X$, then (\ref{eqn2}) becomes (after some algebra)
\begin{equation}
st(t^4 + t^2u^2 - 2stuv + s^2v^2)^2(-t^5 + 2st^3u - 2t^3u^2 + 2stu^3 - tu^4 + s^4v - 2s^3uv) = 0.
\end{equation}
But $st(t^4 + t^2u^2 - 2stuv + s^2v^2)=0$ is the locus in indeterminacy for $\psi \circ \phi$, and so we have
$-t^5 + 2st^3u - 2t^3u^2 + 2stu^3 - tu^4 + s^4v - 2s^3uv = 0$.
Solving for $v$ and plugging in to (\ref{phi}), we obtain (\ref{paramagain}).
\end{proof}

\section{The Elliptic Fibration}\label{fibration}

In this section, we show that the fiber $C(a,b)$ of $X \to \Pj^1$ is a genus 1 curve with two singular points, and we find a smooth model $\tilde{C}(a,b)$ in Weierstrass form.  Finally, we classify the torsion points on $\tilde{C}(a,b)$ for every $a,b \in \Z$.

\subsection{Normalization of the Fiber}\label{fibration1}

\begin{prop}\label{curvesing}
The curve $C(az,bz)$ in $\Pj^2$ (coordinates $[x:y:z]$) has singular points $[1:i:0]$ and $[1:-i:0]$.
\end{prop}
\begin{proof}
Let $g(x,y,z)=f(az,bz,x,y)$ where $f$ is the polynomial (\ref{f}) defining $X$, so $C(az,bz)$ has homogeneous equation $g(x,y,z)$.  Then,
\begin{equation}
C(az,bz)^{\rm sing} = \left\{[x:y:z] \in \Pj^2 : g=\frac{\partial g}{\partial x}=\frac{\partial g}{\partial y}=\frac{\partial g}{\partial z}\right\}.
\end{equation}
Solving for $x,y,z$, we find the two singular points $[1:i:0]$ and $[1:-i:0]$.
\end{proof}

\begin{prop}
If $a,b \in \R$, the real points of $C(a,b)$ form a smooth closed loop.
\end{prop}
\begin{proof}
Let $(x,y) \in C(a,b)(\R)$, and let $r = 1+\left(\frac{b}{a}\right)^2$.  By (\ref{eqn2}),
\begin{equation}
x (a^2 + b^2) (a^2 + b^2 - 2 a x - 2 b y) = a (x^2 + y^2) (x^2 + y^2 - 2 a x - 2 b y).
\end{equation}
Suppose $x^2+y^2>r^2(a^2+b^2)$; then, $r^n(a^2+b^2) < x^2+y^2 \leq r^{n+1}(a^2+b^2)$ for some $n \geq 2$.  Thus,
\begin{align}
|x| (a^2 + b^2) (a^2 + b^2 - 2 a x - 2 b y) &> r^n |a| (a^2 + b^2) (a^2 + b^2 - 2 a x - 2 b y) \nn \\
|x| &> r^n |a|.
\end{align}
Thus,
\begin{equation}
r^{2n} a^2 < x^2 \leq x^2+y^2 \leq r^{n+1}(a^2+b^2) = r^{n+2} a^2,
\end{equation}
so $r^{n-2} < 1$.  This is impossible because $r > 1$ and $n \geq 2$.  Thus,
\begin{equation}
x^2+y^2 \leq r^2(a^2+b^2).
\end{equation}
That is, $C(a,b)(\R)$ is contained in the closed disc of radius $r \sqrt{a^2+b^2}$.

Finally, note that $C(a,b)$ has no real singularities, by Proposition \ref{curvesing}.
\end{proof}

\begin{cor}
If $a,b \in \Z$, $C(a,b)(\Z)$ is a finite set.
\end{cor}
\begin{proof}
$C(a,b)(\Z) = C(a,b)(\R) \cap \Z^2$ is a compact discrete subset of $\R^2$, hence finite.
\end{proof}

\begin{figure}[h]\label{peanut}
\begin{center}
\includegraphics[scale=.7]{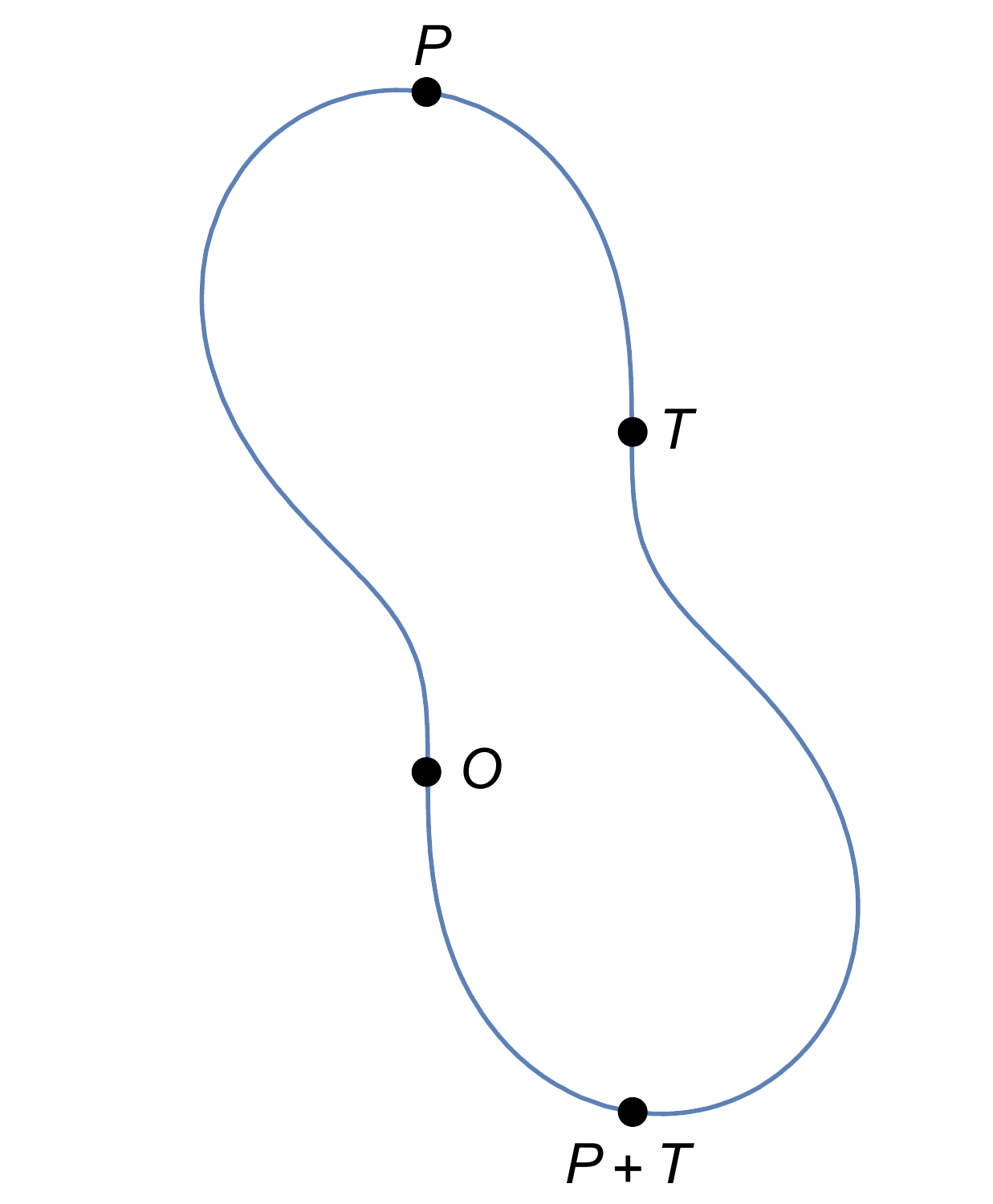}
\end{center}
\caption{The real points of the curve $C(a,b)$.}
\end{figure}

We normalize the $C(a,b)$ to obtain a smooth model, and convert to Weierstrass form.

\begin{prop}\label{ellprop}
The normalization of $C(a,b)$ is an elliptic curve (except when $a=0$) $\tilde{C}(a,b)$ with affine equation
\begin{align}%\label{niceform}
W^2=Z^3+(a^2-2b^2)Z^2+(a^2+b^2)^2Z.
\end{align}
In the standard Weierstrass form, its equation is
\begin{align}\label{weierform}
Y^2=X^3+\frac{1}{3}(2a^4 + 10a^2b^2 - b^4)X-\frac{1}{27}(a^2 - 2b^2)(7a^4 + 26a^2b^2 + b^4).
\end{align}
\end{prop}
\begin{proof}
Define the rational function $\phi$ by
\begin{align}\hspace{-.04in}
\phi(x,y) = \left(\frac{-(a^2 + b^2)^2 (a-x)}{(a^2 + b^2) x - a(x^2 + y^2)},\frac{a (a^2 + b^2)^2 (a-x) (a^2 + b^2 - a x - b y + x^2 + y^2)}{(b x - a y) ((a^2 + b^2) x - a(x^2 + y^2))}
\right).
\end{align}
It is a straightforward computation to check that $(Z,W)=\phi(x,y)$ satisfies (\ref{niceform}) for generic $(x,y) \in C(a,b)$, and that $\phi$ is invertible on a Zariski open set whose complement does not contain $X$.  The discriminant of the RHS of (\ref{niceform}) is computed to be
\begin{equation}
\Delta = -48a^2(a^2+b^2)^4(a^2+4b^2),
\end{equation}
so (\ref{niceform}) defines a nonsingular genus one curve (unless $a=0$). 

The standard Weierstrass form comes from the substitution $(X,Y)=(Z+\fot (a^2-2b^2),W)$.
\end{proof}

We are now ready to prove Theorem \ref{ellthm}, which we restate here for convenience.
\begin{thm}
$X$ is a (singular) rational elliptic surface.  The map $X \to \Pj^1$ given in homogeneous coordinates by $[x:y:a:b] \mapsto [a:b]$ (and having fibers $C(a,b)$), along with the choice of identity section $[a:b] \mapsto [0:0:a:b]$, is an elliptic fibration.  The singular fibers over $\ol{\Q}$ occur at $[a:b] = [0:1], [\pm i : 1], [\pm 2i : 1]$.
\end{thm}
\begin{proof}
This result follows from Proposition \ref{ellprop}.  All that needs to be checked is that the discriminant of $\tilde{C}(a,b)$ is
\begin{equation}
\Delta = -48a^2(a^2+b^2)^4(a^2+4b^2),
\end{equation}
which vanishes precisely at $[a:b] = [0:1], [\pm i : 1], [\pm 2i : 1]$.
\end{proof}

There are four ``trivial'' integer points on $C(a,b)$, giving rise to physically uninteresting solutions to the CHME. The point $O=(0,0)$ is taken to the identity of $\tilde{C}(a,b)$, while $T=(a,b)$ has order 2.  Physically, these torsion points are single Rossby waves (considered as ``triads'' along with two waves of amplitude zero).  The points $P=(0,2b)$ has infinite order, and $(a,-b)=P+T$.  Physically, $P$ and $P+T$ are zonal resonances.

The involution $Q \mapsto Q+P$ is represented by rotation about the center in Figure \ref{peanut}.  However, the negation involution $Q \mapsto -Q$ (equivalently, $(Z,W) \mapsto (Z,-W)$ in Equation \ref{niceform}) looks much more complicated in the $(x,y)$-coordinates, and takes $P$ and $P+T$ to nontrivial points on some higher-level $C(\lambda a, \lambda b)$.

\subsection{Torsion Points on the Fiber}

We will now rule out additional torsion points on $\tilde{C}(a,b)(\Q)$.  This is a rather involved computation and will occupy the next several pages.

\begin{thm}\label{torsion}
For every $a,b\in\Z$ with $a \neq 0$, the torsion of $\tilde{C}(a,b)(\Q)$ is $\Z/2\Z$, generated by $T$.
\end{thm}
\begin{proof}
By a famous theorem of Barry Mazur \cite{maz1, maz2}, the torsion group is one of the cyclic groups $C_m$ for $1 \leq m \leq 10$ or $m=12$, or $C_2 \times C_{2n}$ for $1 \leq n \leq 4$.  There can be no $7$-torsion because there is $2$-torsion.  So, it is sufficient to rule out additional $2$-torsion, $3$-torsion, $4$-torsion, and $5$-torsion.  This will be done by the method of division polynomials.  The $3$-torsion is by far the most difficult to rule out, and the hard work is carried out in Lemma \ref{hyperelllem}.

Let $\psi_n$ denote the $n^{\rm th}$ \textit{division polynomial} for an elliptic curve $y^2=x^3+Ax+B$ in Weierstrass form.  For odd $n$, $\psi_n$ is a polynomial of $A$, $B$, and $x$; for even $n$, it is $y$ times a polynomial of $A$, $B$, and $x$.  The solutions of $\psi_n(x,y) = 0$ give the locations of the $n$-torsion points of the elliptic curve.  The \textit{primitive division polynomial} $\hat{\psi}_n = \prod_{d|n} \psi_d^{\mu(n/d)}$ is a polynomial of $A$, $B$, and $x$ (for $n \neq 2$), and the solutions of $\hat\psi_n(x,y) = 0$ gives the $x$-values of the primitive $n$-torsion points, each of which has two $y$-values.  For more information on division polynomials (including a recursive definition), see \cite{Schoof, LangElliptic}, as well as exercise 3.7 in \cite{Sil1}.

\textbf{The 2-torsion:}
Use the equation (\ref{niceform}) for $\tilde{C}(a,b)$,
\begin{equation}
y^2 = x(x^2+(a^2-2b^2)x+(a^2+b^2)^2).
\end{equation}
The 2-torsion points occur when $y=0$, and the discriminant of the quadratic factor on the right hand side is $\Delta = -3a^2(a^2+4b^2) < 0$.  So the other two torsion points are non-real, and thus irrational.

\textbf{The 3-torsion:}
If we set $\hat{a}=a^2$ and $\hat{b}=b^2$, then $\psi_3$ may be written as
\begin{equation}
\psi_3 = 27x^4 + 18 (2 \hat{a}^2 + 10 \hat{a} \hat{b} - \hat{b}^2)x^2 - 4 (\hat{a} - 2 \hat{b}) (7 \hat{a}^2 + 26 \hat{a} \hat{b} + \hat{b}^2) - (2 \hat{a}^2 + 10 \hat{a} \hat{b} - \hat{b}^2)^2.
\end{equation}
The curve $\{\psi_3(a,b,x)=0\} \subseteq \Pj(1,1,2)$ is a fourfold ramified cover of the curve $\{\psi_3(\hat{a},\hat{b},x)=0\} \subseteq \Pj^2$, and this latter curve is birational to $\Pj^1$ (with homogeneous coordinates $[k:\ell]$).  
\begin{align}
[\hat{a},\hat{b},x] &\mapsto [\hat{b}+x:4\hat{a}+\hat{b}-3x],\\
[k,\ell] &\mapsto [64 k \ell^3 :
2187 k^4 - 162 k^2 \ell^2 - 40 k \ell^3 - \ell^4 : (27 k^2 - 18 k \ell - \ell^2)^2].
\end{align}
Let $(x,y)$ be a rational 3-torsion point on $\tilde{C}(a,b)$.  From the latter birational isomorphism, we obtain
\begin{align}
a^2 &= 64\lambda k \ell^3, \\
b^2 &= \lambda(2187 k^4 - 162 k^2 \ell^2 - 40 k \ell^3 - \ell^4), \\
x     &= \lambda(27 k^2 - 18 k \ell - \ell^2)^2,
\end{align}
for relatively prime integers $k, \ell$ and some $\lambda \in \Q$.  Dividing the first two equations, we obtain
\begin{equation}
\left(\frac{8bk}{a\ell}\right)^2 = 2187 \left(\frac{k}{\ell}\right)^5 - 162 \left(\frac{k}{\ell}\right)^3 - 40 \left(\frac{k}{\ell}\right)^2 - \left(\frac{k}{\ell}\right).
\end{equation}
This looks like the equation of a hyperelliptic curve, which we prove has only two rational points in Lemma \ref{hyperelllem}.  By that lemma, $[k:\ell]$ is either $[0:1]$ or $[1:0]$.  Thus, $a = 64\lambda k \ell^3 = 0$, a contradiction.  So there are no rational $3$-torsion points.

\textbf{The 4-torsion:}
The primitive $4$-division polynomial factors as
\begin{equation}
\hat{\psi}_4 = \frac{\psi_4}{\psi_2} = \frac{1}{729}(3x-(4a^2+b^2))(3x+(2a^2+5b^2))p(x),
\end{equation}
where
\begin{align}
p(x) = 81x^4&+54(a^2 - 2b^2)x^3+54(7a^4 + 26a^2b^2 + b^4)x^2\\
&-6(a^2 - 2b^2)(20a^4 + 82a^2b^2 - b^4)x\\
&+(76a^8 + 364a^6b^2 + 366a^4b^4 + 484a^2b^6 + b^8).
\end{align}
There are four possibly rational $4$-torsion points $R$, those such that $2R=T$, where $T$ is the nontrivial rational $2$-torsion point.  Each such point must be defined over a quadratic extension of $\Q(a,b)$; therefore, they must come from the first two factors of $\hat{\psi}_4$.  These points are
\begin{align}
(x,y) &= \left(\frac{1}{3}(4a^2+b^2),\pm\sqrt{3}a(a^2+b^2)\right) \mbox{ and} \\
(x,y) &= \left(-\frac{1}{3}(2a^2+5b^2),\pm i (a^2+b^2)\sqrt{a^2+4b^2}\right).
\end{align}
Both have irrational $y$-values for all $a,b \in \Q$.

\textbf{The 5-torsion:}
Setting $c=a^2-x$ and $d=b^2+x$, $\psi_5$ may be written as
\tiny
\begin{align}
3^{12}\psi_5 = \ & 23328 (c - 2 d)^5 (82 c^4 + 316 c^3 d + 510 c^2 d^2 + 292 c d^3 + 
   97 d^4)x^3 \\ & - 1296 (c - 2 d)^4 (2740 c^6 + 21552 c^5 d + 60396 c^4 d^2 + 
   79676 c^3 d^3 + 40623 c^2 d^4 + 12894 c d^5 + 157 d^6)x^2 \\ & + 36 (c - 2 d)^3 (69296 c^8 + 824032 c^7 d + 3756944 c^6 d^2 + 
   8338240 c^5 d^3 + 9252272 c^4 d^4 \\ & \hspace{1.173in} + 4279360 c^3 d^5 + 
   1262552 c^2 d^6 + 40744 c d^7 + 803 d^8)x \\ & - (430016 c^{12} + 3828480 c^{11} d - 1707456 c^{10} d^2 - 125933632 c^9 d^3 - 
 526043088 c^8 d^4 \\ & \hspace{.7in} - 830657952 c^7 d^5 - 223114416 c^6 d^6 + 
 677936160 c^5 d^7 + 406003284 c^4 d^8 \\ & \hspace{.7in} + 155427176 c^3 d^9 + 
 7889280 c^2 d^{10} + 216348 c d^{11} - 409 d^{12})
\end{align}

\normalsize
The curve $\{\psi_5(c,d,x)=0\} \subseteq \Pj^2$ may be checked (by computer, in Magma) to be geometric genus 1, with singular points $[c:d:x] = [0:0:1],[-1:4:1],[-5:2:1]$.  An elliptic model is given by
\begin{equation}
Y^2 Z + YZ^2 = X^3 - X^2 Z,
\end{equation}
with $[0:0:1] \mapsto [0:1:0] = \infty$.
This elliptic curve has Mordell-Weil group $\Z/5\Z$ (as computed by Magma).  The five rational points are: $[0:1:0]$, $[0:0:1]$, $[1:-1:1]$, $[1:0:1]$, and $[0:-1:1]$.  They correspond to the points $[0:0:1]$, $[-5:2:1]$, $[-1:4:1]$, $[-5:2:1]$, and $[-1:4:1]$ on $\{\psi_5(c,d,x)=0\}$, respectively.  (We see $[-1:4:1]$ and $[-5:2:1]$ twice because we've resolved a cuspidal singularity at $[0:0:1]$ and nodes at $[-1:4:1]$ and $[-5:2:1]$.)
If $[c:d:x]=[0:0:1]$, then $[a:b]=[\pm i:1]$, an irrational point.  If $[c:d:x]=[-1:4:1]$, then $[a:b]=[0:1]$, and $a=0$  If $[c:d:x]=[-5:2:1]$, then $[a:b]=[\pm 2i:1]$, an irrational point.  
\end{proof}

\begin{cor}
The only resonant Rossby wave triads corresponding to torsion points on any $\tilde{C}(a,b)$ are single wave solutions to the CHME.
\end{cor}
\begin{proof}
This result is just a restatement of Theorem \ref{torsion}.
\end{proof}

The $3$-torsion case comes down to determining the exact set of rational points on a particular hyperelliptic curve.  This requires more advanced methods than are used in the rest of this paper.  An overview of the methods available for enumerating rational points on curves is given by Stoll \cite{S}, whereas the method of Chabauty and Coleman is detailed in \cite{MP}.  There are also the original papers of Chabauty \cite{Ch} (in French) and Coleman \cite{Co1,Co2}.  Finally, we use the computer algebra system Magma for some minor computations; the Magma handbook \cite{magma} describes the algorithms used, primarily $n$-descent as described by Stoll \cite{S}.

\begin{lem}\label{hyperelllem}
The hyperelliptic curve $C$ with affine equation $y^2 = 2187 x^5 - 162 x^3 - 40 x^2 - x$ has only two rational points, $Q_1 = (0,0)$ and $Q_2 = \infty$.
\end{lem}
\begin{proof}
In the first part of the proof, we will bound the number of rational points in each congruence class modulo 7 by the method of Chabauty and Coleman.  In the second part of the proof, we rule out some points using a descent argument.

We begin with the method of Chabauty and Coleman.  The discriminant of $2187 x^5 - 162 x^3 - 40 x^2 - x$ is $-2^{20}3^{17}$, so $C$ has good reduction away from $2$ and $3$.  We use the prime $p=7$ throughout.

Let $J=\Jac(C)$, $r$ be the rank of $J(\Q)$, and $g$ the genus of $C$.  Chabauty will apply if $r < g$.  The genus of a hyperelliptic curve of odd degree $d$ is $(d-1)/2$, so $g=2$.  A computation in Magma shows that $r=1$.  Specifically, \texttt{RankBounds(J)} returns $0 \leq r \leq 1$, \texttt{TorsionSubgroup(J)} returns $\Z/2\Z$, and \texttt{Points(J:Bound:=1000)} returns several points including $P:=[Q_1-Q_2]$ and $R:=[(\frac{1}{27}(-4+\sqrt{-11}),\frac{2}{81}(7+5\sqrt{-11}))+(\frac{1}{27}(-4-\sqrt{-11}),\frac{2}{81}(7-5\sqrt{-11}))-2Q_2]$.  The point $P$ has order $2$, so $R$ has infinite order. 

$C(\F_7)$ has four rational points: $(0,0)$, $\infty$, $(6,1)$, and $(6,6)$.  A straightforward application of Coleman's theorem (Theorem 5.3 of \cite{MP}) gives
\begin{equation}
\#C(Q) \leq \#C(\F_7) + (2g-2) = 6.
\end{equation}
To show $\#C(Q)=2$, we have to do more work.

The point $P+42R = [Q_1' + Q_2' - 2Q_2]$, where $Q_1'$ and $Q_2'$ are defined over $\Q_7$.
\begin{align}
Q_1' &= (5(7)^2+1(7)^3+1(7)^4+\cdots, 4(7)+3(7)^2+1(7)^3+\cdots),\\
Q_2' &= (6(7)^{-2}+0(7)^{-1}+3+\cdots,5(7)^{-5}+0(7)^{-4}+2(7)^{-3}+\cdots).
\end{align}

Let $\omega$ be a differential satisfying (i)-(iii) of sections 5.1 and 5.4 of \cite{MP}; $\omega$ is a $\Z_7$-linear combination of $\frac{dx}{y}$ and $\frac{xdx}{y}$ satisfying
\begin{equation}\label{omegaeq}
\int_{Q_1}^{Q_1'} \omega + \int_{Q_2}^{Q_2'} \omega = 0.
\end{equation}
(For background about $p$-adic integration, see \cite{MP} and the references therein.)
Because $Q_1$ and $Q_1'$ have the same reduction modulo 7, as have $Q_2$ and $Q_2'$, the relevant $p$-adic integrals may be computed by antidifferentiation.
\begin{align}
\int_{Q_1}^{Q_1'}\frac{dx}{y} &= \int_0^{4(7)+3(7)^2+1(7)^3+\cdots}(-2-160y^2-18228y^4-\cdots)dy \\
&= \left.\left(-2y-\frac{160}{3}y^3-\frac{18228}{5}y^5-\cdots\right)\right|_{y=0}^{4(7)+3(7)^2+1(7)^3+\cdots}\\
&= 6(7)+6(7)^2+1(7)^3+\cdots.\\
\int_{Q_1}^{Q_1'}\frac{xdx}{y} &= \int_0^{4(7)+3(7)^2+1(7)^3+\cdots}(2y^2+240y^4+30704y^6+\cdots)dy \\
&= \left.\left(\frac{2}{3}y^3+48y^5+\frac{30704}{7}y^7+\cdots\right)\right|_{y=0}^{4(7)+3(7)^2+1(7)^3+\cdots}\\
&= 3(7)^3+6(7)^4+6(7)^5+\cdots.
\end{align}
Local coordinates at $\infty$ are given by $z=\frac{1}{x}$ and $w=\frac{y}{x^3}$.  We have $\frac{dx}{y}=-\frac{zdz}{y}$ and $\frac{xdx}{y}=-\frac{dz}{y}$. ...
\begin{align}
\int_{Q_2}^{Q_2'}\frac{dx}{y} &= \int_0^{2(7)+5(7^2)+4(7^3)+\cdots}(-2\cdot3^{-14}w^2-16\cdot3^{-31}w^6-\cdots)dw \\
&= \left.\left(-2\cdot3^{-15}w^3-\frac{16}{7}\cdot3^{-31}w^7-\cdots\right)\right|_{w=0}^{2(7)+5(7^2)+4(7^3)+\cdots}\\
&= 2(7)^3+1(7)^4+0(7)^5+\cdots.\\
\int_{Q_2}^{Q_2'}\frac{xdx}{y} &= \int_0^{2(7)+5(7^2)+4(7^3)+\cdots}(-2\cdot3^{-7}-4\cdot3^{-23}w^4-\cdots)dw \\
&= \left.\left(-2\cdot3^{-7}w-\frac{4}{5}\cdot3^{-23}w^5-\cdots\right)\right|_{w=0}^{2(7)+5(7^2)+4(7^3)+\cdots}\\
&= (7)+2(7)^2+0(7)^3+\cdots.
\end{align}
Adding,
\begin{align}
\int_{Q_1}^{Q_1'} \frac{dx}{y} + \int_{Q_2}^{Q_2'} \frac{dx}{y} &= 6(7)+6(7)^2+3(7)^3+\cdots \\
\int_{Q_1}^{Q_1'} \frac{xdx}{y} + \int_{Q_2}^{Q_2'} \frac{xdx}{y} &= (7)+2(7)^2+3(7)^3+\cdots
\end{align}
By (\ref{omegaeq}), we must have (up to an irrelevant scalar in $\F_7^\times$)
\begin{equation}
\omega \con \frac{dx}{y}+\frac{xdx}{y} \Mod{7}.
\end{equation}
The differential $\omega \con \frac{(1+x)dx}{y} = -\frac{(z+1)dz}{w}$ has order of vanishing $m=0$ at $(0,0)$ and $\infty$, and $m=1$ at $(6,1)$ and $(6,6)$.  By part (a) of Coleman's theorem (Theorem 5.3 of \cite{MP}), $C(\Q)$ contains at most one point in each of the congruence classes of $(0,0)$ and $\infty$, and at most two points in each of the congruence classes of $(6,1)$ and $(6,6)$.  Since we've already found $Q_1=(0,0)$ and $Q_2=\infty$, all that remains is to rule out the classes of $(6,1)$ and $(6,6)$.

Assume $x \con 6 \Mod{7}$.  By factoring the RHS of the equation of $C$,
\begin{equation}
y^2=x(2187 x^4 - 162 x^2 - 40 x - 1),
\end{equation}
and a divisibility argument shows that any rational solution must have either $x$ and $(2187 x^4 - 162 x^2 - 40 x - 1)$ both squares, or both $3$ times squares.  But $\left(\frac{6}{7}\right)=-1$, so $x$ is not a square.  Thus, writing $(2187 x^4 - 162 x^2 - 40 x - 1)=v^2/3$ and $x=u/3=w^2/3$,
\begin{equation}
v^2=81 u^4 - 54 u^2 - 40 u - 3.
\end{equation}
Call $D$ the smooth completion of the curve defined by this equation.  Under the birational transformation $\phi : D \simto E$ given by
\begin{align}
s&=\frac{9}{8}(9u^2 - v - 1),\\
t&=\frac{1}{16}(729 u^3 - 243 u - 81uv - 98),
\end{align}
this equation becomes the elliptic curve $E$ defined by
\begin{equation}
t^2+t=s^3+20.
\end{equation}
The Mordell-Weil group of $E$ is generated by the points $(0,4)$, of order 3, and $(-2,3)$, of infinite order.  $E(\F_7)$ has order $12$, but $(0,4)$ and $(-2,3)$ reduce to points of order $3$ and $2$, respectively.  So, $\ol{E(\Q)}$ is not all of $E(\F_7)$; specifically,
\begin{equation}
\ol{E(\Q)} = \{\infty,(4,6),(0,4),(5,3),(0,2),(4,0)\}.
\end{equation}
Converting back to coordinates on $D$ via $\phi^{-1}$,
\begin{equation}
\ol{D(\Q)} = \{\infty_-,\infty_+,(6,3),(3,4),(3,2),(6,0)\}.
\end{equation}
But we've restricted to the case where $x \con 6 \Mod{7}$, that is, $u = 3x \con 4 \Mod{7}$, so it's impossible to have $u \con \infty, 3, 6 \Mod{7}$, and there are no additional rational points on $C$.
\end{proof}

\section{Special Case: Resonant Wavevectors with Zonal Group Velocity Zero}\label{egz}

Rossby waves have phase velocity vector $C$ and group velocity vector $C_g$, where
\begin{align}
C &= (C_x,C_y) = \left(\frac{\omega}{k},\frac{\omega}{\ell}\right) = \left(-\frac{\beta}{k^2+\ell^2},-\frac{\beta k/\ell}{k^2+\ell^2}\right),\\
C_g &= (C_{gx},C_{gy}) = \left(\frac{\partial \omega}{\partial k},\frac{\partial \omega}{\partial \ell}\right) = \left(\frac{\beta\left(k^2-\ell^2\right)}{\left(k^2+\ell^2\right)^2},\frac{2\beta k \ell}{\left(k^2+\ell^2\right)^2}\right).
\end{align}
See \cite{Ped} section 3.19 for details.

Assume $k\ell \neq 0$.  
The fact that the zonal component of $C$ are negative means that crests of Rossby waves travel westward.  
However, the zonal component $C_{gx}$ of $C_g$ can be either positive, zero, or negative.  If $C_{gx} = 0$, wave packets stay in the same longitude (propagating either due northward or due southward).

We ask: What resonant wavevectors have zonal group velocity zero?
All wavevectors $(k,\ell)$ with $C_{gx}=0$ are exactly those with $k^2-\ell^2=0$, that is, $k=\pm \ell$.  By symmetry, it suffices to consider $(k,\ell)=(n,n)$ with $n>0$.

On $C(1,1)$, there are only the four trivial integer points $0=(0,0)$, $T=(1,1)$, $P=(0,2)$, and $T+P=(1,-1)$.  By a computation in Magma, we find that $C(1,1)$ has rank $1$, and its Mordell-Weil group is generated by $T$ and $P$.

\begin{figure}[h]
\renewcommand{\arraystretch}{1.5}
\hspace{-.25in}
\begin{tabular}{c|c|c|c}
point & $(Z,W)$ & $(x,y)$ exactly & $(x,y)$ numerically \\
\hline
$0$ & $\infty$ & $(0,0)$ & $(0.000,0.000)$ \\
$T$ & $(0,0)$ & $(1,1)$ &  $(1.000,1.000)$ \\ 
\hline
$P$ & $\left(1,2\right)$ & $\left(0,2\right)$ & $(0.000,2.000)$ \\
$P+T$ & $\left(4,-8\right)$ & $\left(1,-1\right)$ & $(1.000,-1.000)$ \\ 
$-P$ & $\left(1,-2\right)$ & $\left(\frac{16}{13},\frac{2}{13}\right)$ & $(1.231,0.154)$ \\
$-P+T$ & $\left(4,8\right)$ & $\left(-\frac{3}{13},\frac{11}{13}\right)$ & $(-0.231,0.846)$ \\ 
\hline
$2P$ & $\left(\frac{9}{16},-\frac{93}{64}\right)$ & $\left(\frac{256}{229},\frac{88}{229}\right)$ & $(1.118,0.384)$ \\
$2P+T$ & $\left(\frac{64}{9},\frac{496}{27}\right)$ & $\left(-\frac{27}{229},\frac{141}{229}\right)$ & $(-0.118,0.616)$ \\ 
$-2P$ & $\left(\frac{9}{16},\frac{93}{64}\right)$ & $\left(\frac{1792}{3277},\frac{6568}{3277}\right)$ & $(0.547,2.004)$ \\
$-2P+T$ & $\left(\frac{64}{9},\frac{496}{27}\right)$ & $\left(\frac{1485}{3277}, -\frac{3291}{3277}\right)$ & $(0.453,-1.004)$ \\ 
\hline
$3P$ & $\left(\frac{3025}{49},-\frac{165110}{343}\right)$ & $\left(\frac{5488}{504613}, -\frac{147742}{504613}\right)$ & $(0.011,-0.293)$\\
$3P+T$ & $\left(\frac{196}{3025},\frac{84056}{166375}\right)$ & $\left(\frac{499125}{504613}, \frac{652355}{504613}\right)$ & $(0.989,1.293)$ \\
$-3P$ & $\left(\frac{3025}{49},\frac{165110}{343}\right)$ & $\left(-\frac{1020768}{155870857}, \frac{35570402}{155870857}\right)$ & $(-0.007,-0.228)$\\
$-3P+T$ & $\left(\frac{196}{3025},-\frac{84056}{166375}\right)$ & $\left(\frac{156891625}{155870857}, \frac{120300455}{155870857}\right)$ & $(1.007,0.772)$ \\ 
\hline
$4P$ & $\left(\frac{889249}{553536},\frac{1164191023}{411830784}\right)$ & $\left(-\frac{11531261952}{34589637433}, \frac{55149444912}{34589637433}\right)$ & $(-0.333,1.594)$ \\
$4P+T$ & $\left(\frac{2214144}{889249}, -\frac{3675053536}{838561807}\right)$ & $\left(\frac{46120899385}{34589637433}, -\frac{20559807479}{34589637433}\right)$ & $(1.333,-0.594)$ \\
$-4P$ & $\left(\frac{889249}{553536},-\frac{1164191023}{411830784}\right)$ & $\left(\frac{79003970279424}{58803854910601}, -\frac{9603665729328}{58803854910601}\right)$ & $(1.343,-0.163)$ \\
$-4P+T$ & $\left(\frac{2214144}{889249}, \frac{3675053536}{838561807}\right)$ & $\left(-\frac{20200115368823}{58803854910601}, \frac{68407520639929}{58803854910601}\right)$ & $(-0.344,1.163)$
\end{tabular}
\caption{The first few rational points on $C(1,1)$, computed in Magma.}
\label{figtab}
\end{figure}

The sequence $13, 229, 3277, 504613, 155870857, 34589637433, 58803854910601, \ldots$ of denominators of the $(x,y)$ in Figure \ref{figtab} are the natural numbers $n$ so that $(n,n) \in \Lambda'$.  The resonant wavevectors with zonal group velocity zero are those of the form $(mn,mn)$ and $(mn,-mn)$, where $m \in \Z$ and $n$ is in this sequence.
To our knowledge, we have the first efficient method for computing this sequence, and this paper is the first place in the literature where it has been computed beyond $3277$.

\section{Acknowledgments}

Thank you to Jeffrey C. Lagarias for referring the author to the paper of Kishimoto and Yoneda \cite{KY}, and to Charles R. Doering and Jeffrey C. Lagarias for reading drafts of this paper.  Thank you to Miguel Bustamante, Charles R. Doering, Zaher Hani, Jeffrey C. Lagarias, and Wei Ho for helpful conversations and correspondences.

%%%%%%%%%%%%%%%%%%%%

\bibliographystyle{siam}{}
\bibliography{references}

%%%%%%%%%%%%%%%%%%%%

\end{document}